\documentclass{amsart}
\usepackage{amssymb}
\usepackage{amsfonts}

\newtheorem{theorem}{Theorem}
\theoremstyle{plain}

\newtheorem{proposition}{Proposition}

\renewcommand\bigskip\medskip

\begin{document}
\title[]{On certain recurrent and automatic sequences in finite fields}
\author{Alain LASJAUNIAS and Jia-Yan YAO}
\date{\today}

\begin{abstract}
In this work we extend our study on a link between automaticity and
certain algebraic power series over finite fields.
Our starting point is a family of sequences in a finite field of
characteristic $2$, recently introduced by the first author in connection with algebraic continued fractions. By including it in a large family of recurrent sequences in an arbitrary finite field, we prove its automaticity. Then we give a criterion on automatic sequences, generalizing a previous result and this allows us to present new families of automatic sequences in an arbitrary finite field.
\end{abstract}

\subjclass{Primary 11J70, 11T55; Secondary 11B85}
\keywords{finite fields, power series over a finite field, continued
fractions, finite automata, automatic sequences}
\maketitle

\section{Introduction}

The present work is a continuation of our article \cite{LY15} in which we
have addressed a question concerning the automaticity of the sequence of leading
coefficients of partial quotients for certain algebraic power series. To know more about the motivation and the history, the reader can consult the
introduction of \cite{LY15} and the references given there.

Let $\mathbb{F}_{q}$ be the finite field containing $q$ elements, with $
q=p^{s}$ where $p$ is a prime number and $s\geqslant 1$ is an integer. We
denote by $\mathbb{F}(q)$ the field of power series in $1/T$, with
coefficients in $\mathbb{F}_{q}$, where $T$ is a formal indeterminate.
Hence, an element in $\mathbb{F}(q)$ can be written as $\alpha =\sum_{k\leqslant
k_{0}}u(k)T^{k}$ with $k_{0}\in \mathbb{Z}$, and $u(k)\in \mathbb{F}_{q}$
for all integers $k\leqslant k_0$. These fields of power series are analogues of the field of real numbers. As in the real case, it is well known that the sequence of
coefficients of this power series $\alpha $, $(u(k))_{k\leqslant k_{0}}$, is
ultimately periodic if and only if $\alpha $ is rational, i.e., $\alpha\in
\mathbb{F}_{q}(T)$. Moreover and remarkably, due to the rigidity of the formal case, this sequence of coefficients, for all the elements in $
\mathbb{F}(q)$ which are algebraic over $\mathbb{F}_{q}(T)$, belongs to a class of particular sequences introduced by computer scientists. The origin of
the following theorem can be found in the work of Christol \cite{C} (see
also the article of Christol, Kamae, Mend\`{e}s France, and Rauzy \cite
{CKMFR}).

\begin{theorem}[Christol]
Let $\alpha $ in $\mathbb{F}(q)$ with $q=p^s$. Let $(u(k))_{k\leqslant
k_{0}} $ be the sequence of digits of $\alpha$ and $v(n)=u(-n)$ for all
integers $n\geqslant 0$. Then $\alpha $ is algebraic over $\mathbb{F}_{q}(T)$
if and only if the following set of subsequences of $(v(n))_{n\geqslant 0}$
\begin{equation*}
K(v)=\left\{ {(v(p^{i}n+j))_{n\geqslant 0}\,|\,\,i\geqslant 0,\,0\leqslant
j<p}^{i}\right\}
\end{equation*}
is finite.
\end{theorem}

The sequences having the finiteness property stated in this theorem are
called $p$-automatic sequences. A full account on this topic and a very
complete list of references can be found in the book \cite{AS} of Allouche and Shallit.

Concerning algebraic elements in $\mathbb{F}(q)$, a particular subset need to be considered. An irrational element $\alpha $ in $\mathbb{F}(q)$ is called hyperquadratic, if $\alpha ^{r+1}$, $\alpha ^{r}$, $\alpha$, and $1$ are
linked over $\mathbb{F}_{q}(T)$, with $r=p^{t}$ and $t\geqslant 0$ an
integer. The subset of all these elements, noted $\mathcal{H}(q)$, contains the quadratic ($r=1$) and the cubic power series ($r=p$), but also algebraic elements of arbitrary large degree. For
different reasons, $\mathcal{H}(q)$ could be regarded as the analogue of the subset of
quadratic real numbers, particularly when considering the continued fraction algorithm. See \cite{BL} for more information on this notion. An irrational element $\alpha $ in $\mathbb{F}(q)$ can be expanded as
an infinite continued fraction $\alpha =[a_{1},a_{2},\ldots ,a_{n},\ldots ]$,
where the partial quotients $a_{n}$ are polynomials in $\mathbb{F}_{q}[T]$,
all of positive degree, except perhaps for the first one.
The explicit description of continued fractions for algebraic power series
over a finite field goes back to Baum and Sweet \cite{BS1,BS2},
and was carried on ten years later by Mills and Robbins \cite{MR}.
It happens that this continued fraction expansion can be explicitly
given for various elements in $\mathcal{H}(q)$.
This is certainly the case for quadratic power series,
where the sequence of partial quotients is simply
ultimately periodic (as it is for quadratic real numbers).
It was first observed by Mills and Robbins \cite{MR}
that other hyperquadratic elements have also partial
quotients of bounded degrees,
with an explicit continued fraction expansion,
as a famous cubic over $\mathbb{F}_2$
introduced by Baum and Sweet \cite{BS1}. Some of these examples,
belonging to $\mathcal{H}(p)$ with $p\geqslant 5$, are such that
$a_n=\lambda_nT$, for $n\geqslant 1$, with $\lambda_n\in \mathbb{F}_p^*$.
Then Allouche \cite{A} showed that for each example given in \cite{MR},
with $p\geqslant 5$, the corresponding sequence of partial quotients is
automatic. Another case, in $\mathcal{H}(3)$ also given in \cite{MR},
having $a_n=\lambda_nT+\mu_n$, with $\lambda_n,\mu_n\in
\mathbb{F}_3^*$ for $n\geqslant 1$, was treated by Allouche et al. in \cite{ABS}.
Recently we have investigated 
the existence of such hyperquadratic power series, having partial
quotients of degree $1$, in the largest setting with odd characteristic
(see \cite{LY} and particularly the comments in the last section).
However, concerning the cubic power series introduced by Baum and Sweet
in \cite{BS1}, Mkaouar \cite{Mk} showed that the sequence of partial
quotients (which takes only finitely many values) is not automatic (see
also \cite{Y97}). Besides, we know that most of the elements in
$\mathcal{H}(q)$ have partial quotients of unbounded degrees (see the
introduction in \cite{LY15}). Hence, it appears that the link between automaticity and the sequence of partial quotients is not straight.

With each infinite continued fraction in $\mathbb{F}(q)$, we can
associate a sequence in $\mathbb{F}_{q}^{\ast }$ as follows: if $\alpha
=[a_{1},a_{2},\ldots ,a_{n},\ldots ]$ with $a_n\in \mathbb{F}_{q}[T]$, then for all integers $n\geqslant 1$,
we define $u(n)$ as the leading coefficient of the polynomial $a_{n}$.
For several examples in $\mathcal{H}(q)$, we have observed that
this sequence $(u(n))_{n\geqslant 1}$ is automatic. Indeed, a first
observation in this area is the result of Allouche \cite{A} cited above.
Very recently we have described in
\cite{LY15} three other families of hyperquadratic continued fractions
and have shown that the associated
sequences as indicated above are automatic. For an algebraic (even
hyperquadratic) power series, the possibility of describing explicitly
the continued fraction expansion and consequently the sequence
$(u(n))_{n\geqslant 1}$ is sometimes a difficult problem. In this work
we start with such a description given by the first author in \cite{L4} in
characteristic $2$. In the next section, we show that this sequence belongs to a large family of automatic sequences in a finite field . More precisely, we give the explicit algebraic equation satisfied by the generating function attached to each such sequence. In the last section, we generalize an automaticity criterion introduced in our previous work  \cite{LY15} and this allows us, as an application, to present other recurrent and automatic sequences in a finite field, more general than the preceding ones.

\section{A first family of automatic sequences}

The starting-point of the present work is a family of sequences, defined in a finite field of characteristic 2, which are derived from an algebraic continued fraction in power series fields. The proposition stated below is a simplified version of a theorem proved
recently by the first author in \cite{L4}, improving an earlier result \cite[
Proposition 5, p.~556]{LR}. For the effective coefficients of the algebraic equation appearing in this proposition, the reader is refered to \cite{L4}.

\begin{proposition}
\label{thm02}Let $q=2^{s}$ and $r=2^{t}$ with $s,t\geqslant 1$ integers. Let
$\ell \geqslant 1$ be an integer, and $\Lambda _{\ell +2}=(\lambda_{1},
\lambda _{2},\ldots ,\lambda _{\ell},\varepsilon _{1},\varepsilon _{2})\in (
\mathbb{F}_{q}^{\ast })^{\ell +2}$. We define the sequence $(\lambda _{n})_{n\geqslant 1}$
in $\mathbb{F}_{q}^{\ast }$, recursively from the $\ell$-tuple $(\lambda_{1},
\lambda _{2},\ldots ,\lambda _{\ell})$ as follows. For $m\geqslant 0$,
\begin{equation}
\left\{
\begin{array}{ccc}
\lambda _{\ell +rm+1} & = & (\varepsilon _{2}/\varepsilon _{1})\varepsilon
_{2}^{(-1)^{m+1}}\lambda _{m+1}^{r}, \\
\lambda _{\ell +rm+i} & = & (\varepsilon _{1}/\varepsilon _{2})^{(-1)^{i}}
\quad \text{ for }\quad 2\leqslant i\leqslant r.
\end{array}
\right.  \label{eq01}
\end{equation}
Then there exist $(u,v,w,z)\in (\mathbb{F}_{q}[T])^4$, depending on $\Lambda _{\ell +2}$, such that the continued fraction $\alpha =[\lambda _{1}T,\lambda _{2}T,\ldots ,\lambda _{\ell }T,\ldots
,\lambda _{n}T,\ldots ] \in \mathbb{F}(q)$ , satisfies the following algebraic equation 
\begin{equation*}
uX^{r+1}+vX^{r}+wX+z=0.
\end{equation*}
\end{proposition}

We shall prove that the sequence $(\lambda _{n})_{n\geqslant 1}$, introduced in this proposition, is $2$
-automatic. Here again, this underlines the existence of a link between automaticity and certain algebraic continued fractions, mentioned in the introduction. Indeed we are going to prove the automaticity, via Christol theorem, for a larger class of sequences in a finite field including these introduced above. We prove the following theorem. 

\begin{theorem}
\label{thm04}Let $\ell \geqslant 1$ be an integer, $p\geqslant 2$ a prime
number, $q=p^{s}$ and $r=p^{t}$ with $s,t\geqslant 1$ integers. Let $
k\geqslant 1$ be an integer dividing $r$. Let $(\lambda_{1},
\lambda _{2},\ldots ,\lambda _{\ell})$ be a given $\ell$-tuple in $(\mathbb{F}_{q})^l$. We define recursively in $\mathbb{F}_{q}$ the sequence $(\lambda _{n})_{n\geqslant 1}$ as follows. For $m\geqslant 0$,
\begin{equation}
\left\{
\begin{array}{ll}
\lambda _{\ell +1+r(km+i)}=\alpha _{i+1}\lambda _{km+i+1}^{r}, \quad \text{ for }\quad 0\leqslant i<k, \\
\lambda _{\ell +1+rm+j}=\beta _{j}, \quad \text{ for }\quad 1\leqslant j<r,
\end{array}
\right.   \label{eq04}
\end{equation}
where $\alpha _{i+1}$ $(0\leqslant i<k)$ in $\mathbb{F}_{q}^*$ and $\beta _{j}$ $(0\leqslant j<r)$ in $\mathbb{F}_{q}$ are fixed elements. Set $\theta =\sum\limits_{n\geqslant
1}\lambda _{n}T^{-n}$. Then there exist $A,B,C$ in $\mathbb{F}_{q}(T)$, with $C\neq 0$, and $\rho $ in $
\mathbb{F}(q)$ such that
\begin{equation*}
\theta =A+\rho \quad \text{and}\quad \rho =B+C\rho ^{r}.
\end{equation*} Hence $\theta $ is algebraic over $\mathbb{F
}_{q}(T)$, and then the sequence $(\lambda _{n})_{n\geqslant 1}$ is $p$
-automatic.
\end{theorem}
\noindent \textbf{Remark.} \emph{Note that the sequences in $(1)$ correspond, in $(2)$, to the case :$$p=2, k=2, \alpha_1=\varepsilon _{1}^{-1}, \alpha_2=\varepsilon _{2}^2\varepsilon _{1}^{-1}\quad \text{and} \quad \beta_j=(\varepsilon _{2}/\varepsilon _{1})^{(-1)^j}\quad \text{for} \quad j=1,\dots,r-1.$$}
\begin{proof} According to Christol's theorem, the sequence $(\lambda _{n})_{n\geqslant 1}$ is $p$
-automatic if $\theta$ is algebraic. Let us prove that $\theta$ satisfies an algebraic equation of hyperquadratic type. We define two subset of positive integers: $\mathbf{E}=\{\ell +rn+1\mid n\geqslant 0\}$ and $\mathbf{F}=\{\ell +rn+i\mid n\geqslant 0, \quad 2\leq i\leq r\}$. Hence, we have the following partition $\mathbb{N}^*=\{1,\dots,l\}\bigcup\mathbf{F}\bigcup\mathbf{E}$. We define
$$\rho =\sum_{n\in \mathbf{E}}\lambda _{n}T^{-n} \quad \text{and}\quad \rho_1=\sum_{n\in \mathbf{F}}\lambda _{n}T^{-n}.$$ 
Hence, we have $\theta  =\sum_{m=1}^{\ell }\lambda _{m}T^{-m}+\rho_1+\rho$. By the recursive relations $(2)$, we obtain
\begin{eqnarray*}
\rho_1 &=&\sum_{1\leqslant
j<r}\sum_{m\geqslant 0}\lambda _{\ell +1+rm+j}T^{-(\ell +1+rm+j)}=\sum_{1\leqslant
j<r}\sum_{m\geqslant 0}\beta_jT^{-(\ell +1+rm+j)} \\
&=&\sum_{m\geq 0}(\sum_{1\leqslant
j<r}\beta _{j}T^{-\ell -1-j})T^{-rm} \\
&=&(1-T^{-1})^{-r}\sum_{1\leqslant
j<r}\beta _{j}T^{-\ell -1-j}, 
\end{eqnarray*}
since we have $\sum_{m\geq 0}T^{-m}=(1-T^{-1})^{-1}$ in $\mathbb{F}(q)$. Hence we can write 
$$\theta  =\sum_{m=1}^{\ell }\lambda _{m}T^{-m}+(1-T^{-1})^{-r}\sum_{1\leqslant
j<r}\beta _{j}T^{-\ell -1-j}+\rho= A+\rho,$$
with $A\in \mathbb{F}_{q}(T)$.

To simply the notation, we extend the finite sequence $(\alpha _{i})_{1\leqslant i\leqslant k}$ into
a purely periodic sequence of period length $k$, also denoted by $(\alpha _{n})_{n\geqslant
1}$. Similarly from the recursive relations $(2)$, noting that $\mathbf{E}=\{\ell +1+r(km+i)\mid m\geqslant 0, 0\leq i<k\}$ and since $\alpha_{i+1}=\alpha_{km+i+1}$, we obtain
\begin{eqnarray*}
\rho&=&\sum_{n\in \mathbf{E}}\lambda _{n}T^{-n} =\sum\limits_{0\leqslant i<k}\sum\limits_{m\geqslant 0}\lambda _{\ell +1+r(km+i)}T^{-(\ell +1+r(km+i))}\\
&=&\sum\limits_{0\leqslant i<k}\sum\limits_{m\geqslant 0}\alpha
_{km+i+1}\lambda _{km+i+1}^{r}T^{-(\ell +1+r(km+i))}=T^{r-\ell
-1}\sum_{n\geqslant 0}\alpha _{n+1}\lambda _{n+1}^{r}T^{-r(n+1)}.
\end{eqnarray*}
Consequently, and using our partition of $\mathbb{N}^*$, we can write 
\begin{eqnarray*}
T^{\ell+1-r}\rho&=&\sum_{n\geqslant 1}\alpha _{n}\lambda _{n}^{r}T^{-rn}\\
&=&\sum_{m=1}^{\ell }\alpha_m\lambda _{m}^rT^{-rm}+\sum_{n\in \mathbf{E}}\alpha_n\lambda _{n}^rT^{-nr} +\sum_{n\in \mathbf{F}}\alpha_n\lambda _{n}^rT^{-nr}.\quad (3)
\end{eqnarray*}
Since $k$ divides $r$, again by periodicity, we have $\alpha_{\ell+1+rm+j}=\alpha_{\ell+1+j}$. Applying $(2)$, we get
\begin{eqnarray*}
\sum_{n\in \mathbf{F}}\alpha_n\lambda _{n}^rT^{-nr}&=&\sum_{1\leqslant
j<r}\sum_{m\geqslant 0}\alpha_{\ell+1+rm+j}\lambda _{\ell +1+rm+j}^rT^{-r(\ell +1+rm+j)} \\
&=&\sum_{1\leqslant j<r}\sum_{m\geqslant 0}\alpha_{\ell+1+j}\beta_j^rT^{-r(\ell +1+rm+j)}\\
&=&\sum_{m\geq 0}(\sum_{1\leqslant
j<r}\alpha_{\ell+1+j}\beta _{j}^rT^{-r(\ell+1+j)})T^{-mr^2} \\
&=&(1-T^{-1})^{-r^2}\sum_{1\leqslant
j<r}\alpha_{\ell+1+j}\beta _{j}^rT^{-r(\ell+1+j)}.\quad (4) 
\end{eqnarray*}
By periodicity, we also have $\alpha_{\ell+1+r(km+i)}=\alpha_{\ell+1}$. Hence, we obtain 
\begin{eqnarray*}
\sum_{n\in \mathbf{E}}a_n\lambda _{n}^rT^{-nr}&=&\sum\limits_{0\leqslant i<k}\sum\limits_{m\geqslant 0}\alpha_{\ell+1+r(km+i)}\lambda_{\ell+1+r(km+i)}^rT^{-r(\ell+1+r(km+i))}\\
&=&\alpha_{\ell+1}\sum\limits_{0\leqslant i<k}\sum\limits_{m\geqslant 0}\lambda_{\ell+1+r(km+i)}^rT^{-r(\ell+1+r(km+i))}=\alpha_{\ell+1}\rho^r.\quad (5)
\end{eqnarray*}
Combining $(3)$, $(4)$ and $(5)$, we obtain $\rho=B+C\rho^r$ with $B,C$ in $\mathbb{F}_{q}(T)$, where $C=\alpha_{\ell+1}T^{r-\ell-1}$ and 
$$B=T^{r-\ell-1}(\sum_{m=1}^{\ell }\alpha_m\lambda _{m}^rT^{-rm}+(1-T^{-1})^{-r^2}\sum_{1\leqslant j<r}\alpha_{\ell+1+j}\beta _{j}^rT^{-r(\ell+1+j)}).$$
Thus $\rho$ and also $\theta $ are algebraic over $\mathbb{F}_{q}(T)$ and the proof is complete.
\end{proof}

\noindent \textbf{Remark.} \emph{ From a number-theoretic point of view, the sequences described in Proposition 1 are most important because they are associated with an algebraic continued fraction. This association is not relevant for the more general sequences of Theorem 2 as well as for others of a similar type, even more general, considered in the next section. Hence coming back to the sequences $(\lambda _{n})_{n\geqslant 1}$ in a finite field of characteristic 2, defined by $(1)$, a natural question arises: what can be said about the algebraic degree over $\mathbb{F}_{q}(T)$ of the continued fraction $\alpha=[\lambda_1T,\lambda_2T,\dots,\lambda_nT,....]$ ? According to Proposition 1, this degree is in the range $[2,r+1]$ since $\alpha$ is irrationnal and satisfies an algebraic equation of degree $r+1$. It is a classical fact that $\alpha$ is quadratic if and only if the sequence $(\lambda _{n})_{n\geqslant 1}$ is ultimately periodic. Hence  $\alpha$ is quadratic if and only if $\theta$ (the generating function of the sequence introduced in Theorem 2) is rational.}
\newline In the particular and simplest case $r=2$, we are able to give a necessary and sufficient condition to have this rationality. We prove the following.
\begin{proposition} Let $(\lambda _{n})_{n\geqslant 1}$ be the sequence defined in Proposition 1, by $(1)$, assuming that we have $r=2$. Then this sequence is periodic (and purely periodic of period length less or equal to 2) if and only if we have
 $$\lambda _{m}=(\varepsilon_{1}/\varepsilon _{2})\varepsilon _{2}^{((-1)^{l}-(-1)^{m})/2} \quad \text{for}\quad 1\leqslant m\leqslant \ell. $$
\end{proposition}
\begin{proof} We will apply Theorem 2, in the particular case $p=2$, $k=2$ and $r=2$. Hence we have $\theta=A+\rho$ and $\rho=B+C\rho^2$. Let $V\in\mathbb{F}_{q}(T)$ be given. We have $V+\rho=V+B+CV^2+C(V+\rho)^2$. Setting $\sigma=C(V+\rho)$, multiplying by $C$ this last equality we obtain
$$\sigma=CV+CB+(CV)^2+\sigma^2=U+\sigma^2.\eqno{(6)}$$
Applying the formulas in Theorem 2, in our particular case, we have $C=\alpha_{\ell+1}T^{1-\ell}$ and 
$$B=T^{1-\ell}(\sum_{m=1}^{\ell }\alpha_m\lambda _{m}^2T^{-2m}+\alpha_{\ell+2}\beta _{1}^2(1+T)^{-4}T^{-2\ell}).$$
 The sequence $(\alpha_m)_{m\geq 1}$ is $2$-periodic. Indeed, we have 
$$\alpha_m=(\varepsilon_{2}/\varepsilon _{1})\varepsilon _{2}^{(-1)^{m}}\quad \text{for } m\geq 1\quad \text{and }\quad \beta_1=\varepsilon_{1}/\varepsilon _{2}.$$
Hence, we get $\alpha_{l+1}\alpha_{l+2}\beta_1^2=1$. Now we choose $V=\alpha_{\ell+1}^{-1}T^{1-\ell}(T+1)^{-2}$. Accordingly, a straightforward computation shows that
$$U=CV+CB+(CV)^2=T^{2-2\ell}\sum_{m=1}^{\ell }(\alpha_{\ell+1}\alpha_m\lambda _{m}^2+1)T^{-2m}.$$
We set $u_m=\alpha_{\ell+1}\alpha_m\lambda _{m}^2+1$ and we have
$$U=u_1T^{-2l}+u_2T^{-2l-2}+\dots+u_lT^{-4l+2}.$$
Note that, for $m\geq 0$, between $U^{2^m}$ and $U^{2^{m+1}}$, we have a gap of length $2^{m+1}$. Consequently $\sum_{m\geq 0}U^{2^m}$ is irrational in $\mathbb{F}(q)$, since it has arbitrarily long blocks of zeros in the $(1/T)$ power series expansion unless $U=0$. By $(6)$, we have $\sigma=\sum_{m\geq 0}U^{2^m}$. We also have $\theta=A+V+C^{-1}\sigma$. Therefore $\theta \in \mathbb{F}_{q}(T)$ if and only if $u_m=0$ for $1\leq m\leq l$ or equivalently if and only if 
$$\lambda _{m}^2=(\alpha_{\ell+1}\alpha_m)^{-1}=(\varepsilon_{1}/\varepsilon _{2})^2\varepsilon _{2}^{(-1)^{l}-(-1)^{m}} \quad \text{for}\quad 1\leqslant m\leqslant \ell. $$
It can be easily verified that the sequence $(\lambda _{n})_{n\geqslant 1}$ is then $2$-periodic : $\varepsilon_{1},\varepsilon_{1}/\varepsilon _{2},\dots$ or $\varepsilon_{1}/\varepsilon _{2},\varepsilon_{1}/\varepsilon _{2}^2,\dots$ according to the parity of $l$.
So the proof is complete.
\end{proof}
\noindent \textbf{Remark.} \emph{The statement $\theta=A+V+C^{-1}\sigma$ and $\sigma=U+\sigma^2$ was given in \cite{L4}, without proof. Moreover we can observe the following: $\theta$ is rational if and only if $\alpha$ is quadratic or if and only if we have $\alpha=[\lambda_1T,\lambda_2T,\dots,\lambda_1T,\lambda_2T,\dots]$. Hence, if $(\lambda _{n})_{n\geqslant 1}$ is not purely $2$-periodic then $\alpha$ is cubic over  $\mathbb{F}_{q}(T)$. Furthermore, if we define $\omega(T)=[T,T,\dots,T,\dots]$ (which is the analogue in the formal case of the golden number $(1+\sqrt{5})/2=[1,1,\dots,1,\dots]$) then $\alpha$ is quadratic if and only if we have $\alpha(T)=(\lambda_1/\lambda_2)^{q/2}\omega((\lambda_1\lambda_2)^{q/2}T)$.}    

\bigskip

Inspired by the form of the sequences presented in Theorem \ref{thm04}, we shall give below a criterion for
automatic sequences.

\section{A criterion for automatic sequences}

In this work, we consider sequences of the form $v=(v(n))_{n\geqslant 1}$.
Let $r\geqslant 2$ be an integer. Equivalently, the sequence $v$ is $r$
-automatic if its $r$-kernel
\begin{equation*}
K_{r}(v)=\left\{ {(v(r^{i}n+j))_{n\geqslant 1}\,|\,\,i\geqslant
0,\,0\leqslant j<r}^{i}\right\}
\end{equation*}
is a finite set (see Cobham \cite[p.~170, Theorem 1]{C1}, see also Eilenberg
\cite[p.~107, Proposition 3.3]{E}). For more details on automatic sequences,
see the book \cite{AS} of Allouche and Shallit. Recall that all ultimately
periodic sequences are $r$-automatic for all integers $r\geqslant 2$, adding
or chopping off a prefix to a sequence does not change its automaticity (see
\cite[p.~165]{AS}), and that a sequence is $r$-automatic if and only if it is
$r^{m}$-automatic for all integers $m\geqslant 1$ (see \cite[Theorem 6.6.4,
p.~187]{AS}).

For all integers $j,n$ ($0\leqslant j<r$, and $n\geqslant 1$), define
\begin{equation*}
(T_{j}v)(n)=v(rn+j).
\end{equation*}
Then for all integers $n,a\geqslant 1$, and $0\leqslant b<r^{a}$ with $r$
-adic expansion
\begin{equation*}
b=\sum_{l=0}^{a-1}b_{l}r^{l}\qquad (0\leqslant b_{l}<r),
\end{equation*}
with the help of the operators $T_{j}$\thinspace $(0\leqslant j<r)$, we
obtain
\begin{equation*}
v(r^{a}n+b)=(T_{b_{a-1}}\circ T_{b_{a-2}}\circ \cdots \circ T_{b_{0}}v)(n).
\end{equation*}
In particular, we obtain that $v$ is $r$-automatic if and only if all $
T_{j}v $\thinspace $(0\leqslant j<r)$ are $r$-automatic, for we have $
K_{r}(v)=\{v\}\cup \bigcup_{j=0}^{r-1}K_{r}(T_{j}v)$.

The following theorem generalizes Theorem 2 in \cite{LY15}, and can be
compared with a result of Allouche and Shallit (see \cite[Theorem 2.2]{as2}).

\begin{theorem}
\label{thm2}Let $r\geqslant 2$ be an integer. Let $v=(v(n))_{n\geqslant 1}$ be a sequence in a finite set $A$, 
and $\sigma $ a bijection on $A$. Fix an integer $i$ with $0\leqslant i<r$. Then, for all integer $m\geqslant 0$, we have the following statement.
\begin{enumerate}
\item[($i_{m}$)] If $(T_{i}v)(n+m)=\sigma (v(n))$ for all integers $
n\geqslant 1$, and $T_{j}v$ is $r$-automatic for all integers $j$ ($
0\leqslant j<r$) with $j\neq i$, then $v$ is $r$-automatic.
\end{enumerate}
\end{theorem}

\begin{proof}
 Since $A$ is finite and $\sigma $ is a bijection on $A$,
there exists an integer $l\geqslant 1$ such that $\sigma ^{l}=\mathrm{id}
_{A} $, the identity mapping on $A$. In the following we shall show ($i_{m}$
) by induction on $m$. For this, we need only show that $T_{i}v$ is $r$
-automatic under the conditions of ($i_{m}$).

\smallskip \smallskip If $m=0$, then under the conditions of ($i_{0}$), we
have $T_{i}v=\sigma (v)$, and then
\begin{equation*}
K_{r}(T_{i}v)=\{\sigma ^{a}(T_{i}v)\,|\,0\leqslant a<l\}\cup \bigcup
_{\substack{ 0\leqslant b<l  \\ 0\leqslant j<r,j\neq i}}\sigma
^{b}(K_{r}(T_{j}v)),
\end{equation*}
so $K_{r}(T_{i}v)$ is finite, as $T_{j}v$ is $r$-automatic for all integers $
j$ ($0\leqslant j<r$) with $j\neq i$. \smallskip

If $m=1$, then under the conditions of ($i_{1}$), we have $
(T_{i}v)(n+1)=\sigma (v(n))$ for all integers $n\geqslant 1$, and $T_{j}v$
is $r$-automatic for all integers $j$ ($0\leqslant j<r$) with $j\neq i$.

Below we distinguish two cases:

\smallskip \textbf{Case I:} $0\leqslant i\leqslant r-2$. Then for all
integers $n\geqslant 1$, we have
\begin{equation*}
(T_{0}T_{i}v)(n+1)=(T_{i}v)(rn+r)=\sigma (v(rn+r-1))=\sigma ((T_{r-1}v)(n)),
\end{equation*}
hence $T_{0}(T_{i}v)$ is $r$-automatic, since it is obtained from $\sigma
(T_{r-1}v)$ by adding a letter before, and $T_{r-1}v$ is $r$-automatic by
hypothesis, for $r-1\neq i$ .

Let $j$ be an integer such that $1\leqslant j<r$. Then for all integers $
n\geqslant 1$,
\begin{equation*}
(T_{j}T_{i}v)(n)=(T_{i}v)(rn+j)=\sigma (v(rn+j-1))=\sigma ((T_{j-1}v)(n)).
\end{equation*}
Hence if $j\neq i+1$, then $T_{j}(T_{i}v)$ is $r$-automatic, for $j-1\neq i$
, and thus $T_{j-1}v$ is $k$-automatic by hypothesis. Moreover for $j=i+1$,
we have $T_{i+1}(T_{i}v)=\sigma (T_{i}v)$. Note that $T_{j}(T_{i}v)$ is $r$
-automatic for all integers $j$ ($0\leqslant j<r$) with $j\neq i+1$, hence
we can apply ($(i+1)_{0}$) with $T_{i}v$, and we obtain that $T_{i}v$ is $r$
-automatic.

\smallskip \textbf{Case II:} $i=r-1$. Then for all integers $j,n$ ($
1\leqslant j<r$ and $n\geqslant 1$), we have
\begin{equation*}
(T_{j}T_{r-1}v)(n)=(T_{r-1}v)(rn+j)=\sigma (v(rn+j-1))=\sigma
((T_{j-1}v)(n)).
\end{equation*}
So $T_{j}(T_{r-1}v)$ is $r$-automatic, for $j-1\neq i$, and thus $T_{j-1}v$
is $r$-automatic by hypothesis. Moreover for all integers $n\geqslant 1$, we
have
\begin{equation*}
(T_{0}T_{r-1}v)(n+1)=(T_{r-1}v)(rn+r)=\sigma (v(rn+r-1))=\sigma
((T_{r-1}v)(n).
\end{equation*}
Since $T_{j}(T_{r-1}v)$ is $r$-automatic for all integers $j$ ($1\leqslant
j<r$), we can apply ($0_{1}$) proved above with $T_{r-1}v$, and we obtain
that $T_{r-1}v$ is $r$-automatic.

\smallskip Now let $m\geqslant 1$ be an integer, and assume that ($i_{j}$)
holds for all integers $i,j$\ ($0\leqslant i<r$ and $0\leqslant j\leqslant m$
). We shall show that ($i_{m+1}$) holds for all integers $i$ ($0\leqslant
i<r $). Namely, under the conditions that $(T_{i}v)(n+m+1)=\sigma (v(n))$
for all integers $n\geqslant 1$, and $T_{j}v$ is $r$-automatic for all
integers $j$ ($0\leqslant j<r$) with $j\neq i$, we shall show that $T_{i}v$
is $r$-automatic. For this, we distinguish two cases below.

\smallskip Write $m=r[\frac{m}{r}]+a$, with $0\leqslant a<r$ an integer.

\smallskip \textbf{Case I:} $0\leqslant i<r-a-1$. Let $j$ ($0\leqslant j<r$)
be an integer. If $j<a+1$, then for all integers $n\geqslant 1$, we have
\begin{eqnarray*}
(T_{j}T_{i}v)(n+[\frac{m}{r}]+1) &=&(T_{i}v)(rn+r[\frac{m}{r}
]+k+j)=(T_{i}v)(rn+m+r+j-a) \\
&=&\sigma (v(rn+r+j-a-1))=\sigma ((T_{r+j-a-1}v)(n)),
\end{eqnarray*}
hence $T_{j}(T_{i}v)$ is $r$-automatic, since it is obtained from $\sigma
(T_{r+j-a-1}v)$ by adding a prefix of length $[\frac{m}{r}]+1$, and the
latter is $r$-automatic by hypothesis, for we have $j\geqslant 0>i+a+1-r$.
Now assume $j\geqslant a+1$. Then for all integers $n\geqslant 1$, we have
\begin{eqnarray*}
(T_{j}T_{i}v)(n+[\frac{m}{r}]) &=&(T_{i}v)(rn+r[\frac{m}{r}
]+j)=(T_{i}v)(rn+m+j-a) \\
&=&\sigma (v(rn+j-a-1))=\sigma ((T_{j-a-1}v)(n)).
\end{eqnarray*}
If $j\neq i+a+1$, then $T_{j}(T_{i}v)$ is $r$-automatic, since it is
obtained from $\sigma (T_{j-a-1}v)$ by adding a prefix of length $[\frac{m}{r
}]$, and the latter is $r$-automatic by hypothesis, for we have $j-a-1\neq i$
. If $j=i+a+1$, then $(T_{j}T_{i}v)(n+[\frac{m}{r}])=\sigma ((T_{i}v)(n))$,
for all integers $n\geqslant 1$. Note here that we have $[\frac{m}{r}
]\leqslant m$ and $T_{j}(T_{i}v)$ is $r$-automatic for all integers $j$ ($
0\leqslant j<r$) with $j\neq i+a+1$, hence we can apply ($(i+a+1)_{[\frac{m}{
r}]}$) with $T_{i}v$, and we obtain at once that $T_{i}v$ is $r$-automatic.

\smallskip \textbf{Case II:} $r-a-1\leqslant i<r$. Let $j$ ($0\leqslant j<r$
) be an integer. If $j\geqslant a+1$, then for all integers $n\geqslant 1$,
we have
\begin{eqnarray*}
(T_{j}T_{i}v)(n+[\frac{m}{r}]) &=&(T_{i}v)(rn+r[\frac{m}{r}
]+j)=(T_{i}v)(rn+m+j-a) \\
&=&\sigma (v(rn+j-a-1))=\sigma ((T_{j-a-1}v)(n)),
\end{eqnarray*}
hence $T_{j}(T_{i}v)$ is $r$-automatic, since it is obtained from $\sigma
(T_{j-a-1}v)$ by adding a prefix of length $[\frac{m}{r}]$, and the latter
is $r$-automatic by hypothesis, for we have $i\geqslant r-a-1>j-a-1$. If $
j<a+1$, then for all integers $n\geqslant 1$, we have
\begin{eqnarray*}
(T_{j}T_{i}v)(n+[\frac{m}{r}]+1) &=&(T_{i}v)(rn+r[\frac{m}{r}
]+r+j)=(T_{i}v)(rn+m+r+j-a) \\
&=&\sigma (v(rn+r+j-a-1))=\sigma ((T_{r+j-a-1}v)(n)).
\end{eqnarray*}
If $j\neq i+a+1-r$, then $T_{j}(T_{i}v)$ is $r$-automatic, since it is
obtained from $\sigma (T_{r+j-a-1}v)$ by adding a prefix of length $[\frac{m
}{r}]+1$, and the latter is $r$-automatic by hypothesis, for we have $
r+j-a-1\neq i$. If $j=i+a+1-r$, then $(T_{j}T_{i}v)(n+[\frac{m}{r}
]+1)=\sigma ((T_{i}v)(n))$, for all integers $n\geqslant 1$. Now that $[
\frac{m}{r}]+1\leqslant m$ and $T_{j}(T_{i}v)$ is $r$-automatic for all
integers $j$ ($0\leqslant j<r$) with $j\neq i+a+1$, hence we can apply ($
(i+a+1-r)_{[\frac{m}{r}]+1}$) with $T_{i}v$, \ and we obtain that $T_{i}v$
is $r$-automatic.

\smallskip Finally we conclude that ($i_{m}$) hold for all integers $i,m$\ ($
0\leqslant i<r$ and $m\geqslant 0$).
\end{proof}

As  an application of Theorem~\ref{thm2}, we have the following result
which slightly generalizes Theorem \ref{thm04}.

\begin{theorem}
\label{thm3}Let $\ell \geqslant 1,$ $r\geqslant 2$ and $k\geqslant 1$ be integers such that $
k $ divides $r$. Let $p$ be a prime number and $q=p^s$ where $s\geqslant 1$ is an integer. Let $(u(1),u(2),\dots,u(l))$ be given in $\mathbb{F}_{q}^l$. Let $u=(u(n))_{n\geqslant 1}$ be the sequence in the finite
field $\mathbb{F}_{q}$ such that we have, for all integers $m\geqslant 0$ and $
0\leqslant i<k$,
\begin{equation}
\left\{
\begin{array}{rcl}
u(\ell +1+r(km+i)) & = & \alpha _{i+1}(u(km+i+1))^{\gamma }, \\
u(\ell +1+r(km+i)+j) & = & \beta _{i,j}\text{ }(1\leqslant j<r),
\end{array}
\right.  \label{eq02}
\end{equation}
where $\alpha _{i+1}(0\leqslant i<k)$ in $\mathbb{F}_{q}^*$, $\beta _{i,j} (1\leqslant j<r)$ in $\mathbb{F}_{q}$ are fixed elements and $\gamma \geqslant 1$ is an integer coprime with $q-1$. 
Then the sequence $u$ is $r$-automatic.
\end{theorem}

\begin{proof}
For all integers $i,n\,(0\leqslant i<r,\,n\geqslant 1)$, set $
u_{i}(n)=u(rn+i)$, and we only need to show that all the $u_{i}\,(0\leqslant
i<r)$ are $r$-automatic. \smallskip

Write $\ell +1=ra+b$, with $a,b$ integers such that $a\geqslant 0$, $
0\leqslant b<r$. Then $a+b\geqslant 1$. From the recursive relations (\ref
{eq02}), we deduce at once that all the $u_{j}\,(0\leqslant j<r)$ are
ultimately periodic except for $j=b$, and for all integers $m\geqslant 0$
and $0\leqslant i<k$, we have $u_{b}(km+i+a)=\alpha
_{i+1}(u(km+i+1))^{\gamma }$. Since all the ultimately periodic sequences
are $r$-automatic, there remains for us to show that $u_{b}$ is $r$-automatic.

Extend $(\alpha _{i})_{1\leqslant i\leqslant k}$ to be a periodic sequence
of period $k$, denoted by $(\alpha _{n})_{n\geqslant 0}$. Then for all integers $
n\geqslant 1$, we have
\begin{equation*}
u_{b}(n-1+a)=\alpha _{n}(u(n))^{\gamma },
\end{equation*}
from which, by noting that $k$ divides $r$, we obtain, for all integers $
m\geqslant 1$,
\begin{equation*}
\left\{
\begin{array}{ccl}
u_{b}(rm+a+b-1) & = & \alpha _{rm+b}(u_{b}(m))^{\gamma }=\alpha
_{b}(u_{b}(m))^{\gamma }, \\
u_{b}(rm+a+j-1) & = & \alpha _{rm+j}(u_{j}(m))^{\gamma }=\alpha
_{j}(u_{j}(m))^{\gamma }\text{ }(0\leqslant j<r,\,j\neq b).
\end{array}
\right.
\end{equation*}
Write $a+b-1=rc+d$, with $c,d$ integers such that $c\geqslant 0$, and $
0\leqslant d<r$. Then all the $T_{i}u_{b}\,(0\leqslant i<r)$ are ultimately
periodic (thus $r$-automatic) except for $i=d$, as all the $
u_{j}\,(0\leqslant j<r$, $j\neq b)$ are ultimately periodic. Moreover for
all integers $n\geqslant 1$, we have $(T_{d}u_{b})(m+c)=\alpha
_{b}(u_{b}(m))^{\gamma }=\sigma (u_{b}(m))$, where $\sigma (x)=\alpha
_{b}x^{\gamma }$, and $\sigma $ is bijective on $
\mathbb{F}_{q}$ since $\gamma$ is coprime with $q-1$. To conclude,
it suffices to apply Theorem \ref{thm2} with $u_{b}$ to obtain that $u_{b}$
is also $r$-automatic.
\end{proof}

\noindent \textbf{Remark.} \emph{ In this theorem, if $\gamma$ is a power of $p$ then we are in the case of Theorem 2 and the automaticity follows directly from Christol theorem, as we have seen. In all cases the generating functions of the sequences defined in Theorem 4 are algebraic, due to Christol theorem, but the algebraic equation is not given in the general case and it may not be as simple (hyperquadratic type) as it is in Theorem 2.}

\noindent \medskip

\noindent \textbf{Acknowledgments.} Part of the work was done while Jia-Yan
Yao visited the Institut de Math\'{e}matiques de Jussieu-PRG (CNRS), and he
would like to thank his colleagues, in particular Jean-Paul Allouche, for
their generous hospitality and interesting discussions. He would also like to
thank the National Natural Science Foundation of China (Grants no. 10990012
and 11371210) for partial financial support. 

\vskip 2 cm
\begin{tabular}{ll}
Alain LASJAUNIAS &  \\
Institut de Math\'{e}matiques de Bordeaux &  \\
CNRS-UMR 5251 &  \\
Talence 33405 &  \\
France &  \\
E-mail: Alain.Lasjaunias@math.u-bordeaux.fr &  \\
&  \\
Jia-Yan YAO &  \\
Department of Mathematics &  \\
Tsinghua University &  \\
Beijing 100084 &  \\
People's Republic of China &  \\
E-mail: jyyao@math.tsinghua.edu.cn &
\end{tabular}

\end{document}